\documentclass[10pt,leqno]{amsart}
\date{December 12, 2019}
\usepackage[active]{srcltx}
\usepackage{graphicx} 
\usepackage{amsmath, amssymb, graphics, setspace}
\usepackage{verbatim,enumerate}
\usepackage[dvipsnames]{xcolor}

\title[]{%
Null hypersurfaces 
in Lorentzian manifolds	 
with the null energy condition
}

\author[]{
S.~Akamine, A.~Honda,  M.~Umehara and  
        K.~Yamada 
}   

\address[Shintaro Akamine]{%
Graduate School of Mathematics, 
Nagoya University, Chikusa-ku, Nagoya 464-8602, Japan
}
\email{s-akamine@math.nagoya-u.ac.jp}

\address[Atsufumi Honda]{
Department of Applied Mathematics, 
Faculty of Engineering, Yokohama National University,
79-5 Tokiwadai, Hodogaya, Yokohama 240-8501, Japan
}
\email{honda-atsufumi-kp@ynu.ac.jp}

\address[Masaaki Umehara]{%
   Department of Mathematical and Computing Sciences,
   Tokyo Institute of Technology,
   Tokyo 152-8552, Japan
}
\email{umehara@is.titech.ac.jp}

\address[Kotaro Yamada]{%
   Department of Mathematics,
   Tokyo Institute of Technology,
   Tokyo 152-8551, Japan
}
\email{kotaro@math.titech.ac.jp}


\thanks{
The first and second authors were partially supported by the
Grant-in-Aid for Young Scientists No.~19K14527 and
No.~19K14526, respectively.
The last author
was partially supported by 
(B) No.\  17H02839 from Japan Society for the 
Promotion of Science.
}

\usepackage{amsthm}
\theoremstyle{plain}

 \newtheorem{theorem}{Theorem}[section]
 \newtheorem{proposition}[theorem]{Proposition}
 
 \newtheorem{fact}[theorem]{Fact}
 \newtheorem{lemma}[theorem]{Lemma}

\theoremstyle{definition}
 \newtheorem{definition}[theorem]{Definition}
\theoremstyle{remark}
 
 \newtheorem*{remark*}{Remark}
\newtheorem{example}[theorem]{Example}
 \newtheorem*{acknowledgement}{Acknowledgement}

\numberwithin{equation}{section}

\newcommand{\op}[1]{{\operatorname{#1}}}

\newcommand{\mb}[1]{\vect{#1}}

\newcommand{\vect}[1]{\boldsymbol{#1}}
\newcommand{\R}{\boldsymbol{R}}

\renewcommand{\phi}{\varphi}

\begin{document}
\maketitle

\begin{abstract}
Let $M^{n+1}_1$ be a light-like geodesically  complete
Lorentzian $(n+1)$-manifold satisfying the null energy condition.
We show that null hypersurfaces
properly immersed in $M^{n+1}_1$
are totally geodesic.
\end{abstract}

\section{Introduction} \label{sec:I} 

We fix a $C^\infty$-differentiable 
Lorentzian $(n+1)-$manifold $(M^{n+1}_1,g_L)$.
A $C^2$-immersion $F\colon\Lambda^n\to M^{n+1}_1$ 
defined on an $n$-manifold $\Lambda^n$
is called {\it null} if
the first fundamental form (i.e. the
induced metric $F^*g_L$) is positive 
semi-definite but not positive 
definite at all points.
Such a null hypersurface $F$
is foliated by light-like geodesics
(see Fact \ref{lem:F}), and can be constructed
locally from a given space-like submanifold $S$
of codimension $2$
in $M^{n+1}_1$.
In fact, there are exactly two light-like direction fields
which are normal to $S$,
and two ruled null hypersurfaces in
$M^{n+1}_1$ are generated by these
direction fields.

A continuous map $F\colon\Lambda^n\to M^{n+1}_1$
is said to be {\it proper} if $F^{-1}(K)$ is 
compact for each compact subset $K$ of $M^{n+1}_1$.
On the other hand,
a Lorentzian manifold $M^{n+1}_1$ is
called {\it light-like geodesically complete} 
if any light-like geodesics of $M^{n+1}_1$ 
can be extended to complete geodesics.
The purpose of this paper is to prove the following:

\medskip
\noindent
{\bf Theorem A.}
{\it
Let $M^{n+1}_1$ $(n\ge 2)$ be
a light-like geodesically complete
Lorentzian $(n+1)$-manifold
satisfying the null energy condition $($NEC\,$)$, 
that is,
the Ricci tensor of $M^{n+1}_1$ 
is non-negative along light-like directions.
Suppose that
$F\colon\Lambda^n\to M^{n+1}_1$ 
is a proper null $C^2$-immersion.
Then, $F$ is totally geodesic
$($that is, any geodesic in $M^{n+1}_1$ which is
tangent to $F(\Lambda^n)$ at a certain point
lies in $F(\Lambda^n))$, and
the Ricci tensor of $M^{n+1}_1$ 
vanishes along null directions of $\Lambda^n$.
}

\medskip
Since the inclusion mappings of
topologically closed embedded hypersurfaces
are proper, Theorem A is a generalization of 
Gutierrez and Olea \cite[Proposition 3.11]{GO}, 
where $M^{n+1}_1$ is assumed to have
a time-like conformal vector field.

We also remark that
Galloway \cite[Theorem IV.1]{Ga2} is 
closely related,
which asserts that, if $M^{n+1}_1$
contains an achronal complete light-like geodesic $L$,
then $L$ lies on a totally geodesic
hypersurface.
Theorem A is independent of it,
since the achronality is not assumed in
Theorem A, however a technique in  \cite{Ga2} is applied
as follows: 
A null immersion $F$ is called
{\it $L$-complete} if 
for each $p\in \Lambda^n$,
there exists a curve $\gamma\colon\R\to \Lambda^n$
passing through $p$ such that $F\circ \gamma$  
gives a complete light-like geodesic
in the ambient space $M^{n+1}_1$
(cf. Definition \ref{def:L}).
We first show that 
the  properness
of $F$ implies
$L$-completeness of $F$ (cf. Theorem \ref{Thm:B}).
We then apply 
the well-known Raychaudhuli equation
using  the map as in \cite[Remark IV.1]{Ga2} 
giving a certain splitting-structure of a light-like
hypersurface (cf. Proposition \ref{prop:Phi}).

A Lorentzian space form is a geodesically complete
Lorentzian manifold of constant sectional curvature.
Since any Lorentzian Einstein manifold 
satisfies the null energy condition,
we obtain the following corollary:

\medskip
\noindent
{\bf Corollary B.}
{\it 
A $C^2$-differentiable null 
hypersurface 
which is  properly immersed 
in a Lorentzian Einstein manifold
$($in particular, in a  Lorentzian space form$)$
is totally geodesic.
}

\medskip
We denote by $\R^{n+1}_1$ the 
Lorentz-Minkowski space of 
signature $(-,+,\dots,+)$.
Since an entire graph in $\R^{n+1}_1$ is
the image  of a proper map,
as an application of Corollary B, 
we obtain the following Bernstein-type theorem:

\medskip
\noindent
{\bf Corollary C}\,(\cite{AUY,AHUY}).
{\it
An entire $C^2$-differentiable light-like graph in $\R^{n+1}_1$ is a 
light-like hyperplane.
}

\medskip
Bejancu, Ferr\'{a}ndez and Lucas 
\cite{BFL} showed the same conclusion 
under the assumption that
the graph has zero light-like mean curvature.
So Corollary~C is a generalization of that.
When $n=2$, this assertion has been proved in
the first, third and fourth authors' previous work \cite{AUY}
applying Hartman-Nirenberg's cylinder theorem.
(In fact, null surfaces in $\R^3_1$
are flat surfaces in the Euclidean 3-space,
as pointed out in \cite{AUY}.)
After that, this was proved in \cite{AHUY}
using a method which is different from the one in this paper. 

\medskip
The paper is organized as follows:
In Section \ref{sec:1}, we recall
fundamental properties of
null hypersurfaces and 
prepare two
key assertions
(Proposition \ref{prop:Phi} and Theorem \ref{Thm:B})
to prove Theorem A.
Main results are proved in Section~\ref{sec:main}. 
Several related examples 
are given at the end of Section 3.

\section{$L$-completeness of null hypersurfaces} \label{sec:1} 

\begin{definition}
A $C^2$-immersion 
$
F\colon\Lambda^n \to M^{n+1}_1
$
is said to be {\it null} or {\it light-like} 
if the induced metric $ds^2:=F^*g_L$ is
degenerate everywhere.
Let $U$ be a domain of $\Lambda^n$.
A {\it null vector field} on  
$
U
$
is a vector field $\xi$ without zeros so that
$$
ds^2(\xi,\mb v)=0 \qquad (\mb v\in T\Lambda^n).
$$
\end{definition}
Here, a null vector field may not be defined on the whole of 
$\Lambda^n$ in general.
However, the following assertion can be proved using well-known 
techniques on covering spaces.

\begin{proposition}\label{prop:O}
If $\Lambda^n$ does not admit a null vector field defined on
$\Lambda^n$, then there exists a double covering 
$\pi:\hat \Lambda^n\to \Lambda^n$ so that  
there is a null vector field of $F\circ \pi$ defined on
$\hat \Lambda^n$.
In particular, if  $\Lambda^n$ is simply connected,
we can take a null vector field defined on $\Lambda^n$.
\end{proposition}

So to prove our main results, we may assume that 
$\Lambda^n$ admits a globally defined null vector field. 
By Nomizu and  Ozeki \cite{NO}, there exists a complete Riemannian metric
$h_\Lambda$ on $\Lambda^{n}$, and we now choose that metric. 

\begin{definition}
Let
$
F\colon\Lambda^n \to M^{n+1}_1
$
be a  null immersion.
A null vector field $\xi_\Lambda$ defined on
$\Lambda^n$ 
 is called an {\it $h_\Lambda$-normalized null vector field}
if
$$
h_\Lambda(\xi_\Lambda,\xi_\Lambda)=1.
$$
\end{definition}

Since we assume the existence of a null vector field
defined on $\Lambda^n$, the following assertion follows immediately.

\begin{proposition}\label{lem:L}
Let $F\colon\Lambda^{n}\to M^{n+1}_1$ be a 
$C^2$-differentiable null immersion. 
Then, there exists a $C^1$-differentiable $h_\Lambda$-normalized null 
vector field $\xi_\Lambda$ defined on $\Lambda^{n}$
up to a $\pm$-ambiguity.
\end{proposition}

\begin{definition}
A curve $\hat \gamma\colon J\to M^{n+1}_1$ is a {\it light-like pre-geodesic}
if the image $\hat\gamma(J)$ coincides with the image of a 
light-like geodesic $\sigma$ of $M^{n+1}_1$.
\end{definition}

The following assertion is well-known (see
Galloway \cite[Proposition 3.1]{Ga}
and also \cite[Corollary B]{UY}).

\begin{fact}\label{lem:F}
Let $F\colon\Lambda^n\to M^{n+1}_1$ be a 
$C^2$-differentiable null immersion.
For each $p\in \Lambda^n$, there exist
an open interval $I$ and a $C^2$-regular curve
$\gamma\colon I\to \Lambda^n$ such that
$\gamma(0)=p$ and $\hat \gamma:=F\circ \gamma$ gives a light-like geodesic
of $M^{n+1}_1$.
Moreover, 
for an integral curve $\gamma$ of an $h_\Lambda$-normalized 
null vector field, the curve $\hat{\gamma}=F\circ \gamma$ is a 
light-like pre-geodesic of $M^{n+1}_1$.
\end{fact}

In particular, any null hypersurfaces 
are ruled hypersurfaces foliated by
light-like geodesics.  
Regarding this fact, we define \lq $L$-completeness\rq\ 
as follows:

\begin{definition}\label{def:L}
A $C^2$-differentiable null immersion 
$F\colon\Lambda^n\to M^{n+1}_1$ 
is {\it $L$-complete} if
for each $p\in \Lambda^n$,
there exists a curve $\gamma\colon\R\to \Lambda^n$
passing through $p$ such that $F\circ \gamma$  
gives a complete light-like geodesic
in the ambient space $M^{n+1}_1$.
\end{definition}

We then prove the following assertion, which will play
a crucial role in proving Theorem A:
 
\begin{proposition}\label{prop:Phi}
Let $F\colon\Lambda^n\to M^{n+1}_1$ be a $C^2$-differentiable $L$-complete
null immersion and fix a point $p \in \Lambda^n$.
We let $\Sigma^{n-1}$ 
be an embedded hypersurface of $\Lambda^n$
passing through $p$ which is transversal to a
null vector field.
Then there exists 
an immersion
\begin{equation}\label{eq:gmu}
\Psi:\R\times \Sigma^{n-1} \to \Lambda^n
\end{equation}
such that 
$
\R\ni s \mapsto F\circ \Psi(s,q)\in M^{n+1}_1
$
 is a complete light-like geodesic in $M^{n+1}_1$ for each
$q\in \Sigma^{n-1}$, in particular,
\begin{equation}\label{eq:exp}
F\circ \Psi(s,q)=\op{Exp}^L_{F(q)}(s \hat \xi_q)\qquad 
(s\in \R,\,\, q\in \Sigma^{n-1})
\end{equation}
holds, where $\hat \xi$ is a certain light-like vector field along 
$F(\Sigma^{n-1})$
and $\op{Exp}^L_{F(q)}$ is the exponential map of $M^{n+1}_1$
with respect to $g_L$ centered at $F(q)$.
\end{proposition}

In the case of $M^{n+1}_1=\R^3_1$,
null surfaces in $\R^3_1$  are flat surfaces
with respect to the canonical Euclidean
metric on $\R^3$, as pointed out in the introduction,
and  the  map $\Psi$ just
coincides with the map $\Phi$ given in
Galloway \cite[Remark IV.1]{Ga2}
and Murata-Umehara \cite[Proposition 2.5]{MU}.

\begin{proof}
We let $\xi_\Lambda$ be the $h_\Lambda$-normalized null
vector field on $\Lambda^n$.
Since $h_\Lambda$ is a complete Riemannian metric on $\Lambda^n$,
it is complete as a vector field, and it
induces a 1-parameter group of transformations
$
\Phi\colon\R\times \Lambda^n\to \Lambda^n.
$
By Proposition~A.3 in the appendix,
the following restriction of the map $\Phi$ to 
$\R\times \Sigma^{n-1}\to \Lambda^n$
gives an immersion.
For each $q\in \Sigma^{n-1}$,
the curve given by
$$
\gamma_q:\R\ni t\mapsto \Phi(t,q)\in \Lambda^n
$$
is a maximal integral curve of 
$\xi_\Lambda$ emanating from $q$, and so $F\circ \gamma_q$
gives a light-like pre-geodesic of $M^{n+1}_1$.
It is well-known that
one can find a parameter $s$ defined on an interval $J(\subset \R)$
and a smooth function $t=t(s,q)$
such that $J\ni s\mapsto F(\gamma_q(t(s,q)))\in M^{n+1}_1$
gives a maximal light-like geodesic, and
${\partial t(s,q)}/{\partial s}\ne 0$.
Since $F$ is $L$-complete, the image of the maximal integral curve of 
$\xi_\Lambda$ by $F$ coincides with the image of a complete light-like 
geodesic. Therefore, we may assume that $J=\R$.
So we get an immersion defined by
$$
\Psi:\R\times \Sigma^{n-1}
\ni (s,q)
\mapsto \Phi(t(s,q),q)\in
\Lambda^n,
$$
which has the desired properties, by setting
$$
\hat \xi_q:=
\left. \frac{\partial F(\gamma_q(t(s,q)))}{\partial s}\right|_{s=0}.
$$
\end{proof}

To end this section, we prove the following assertion:

\begin{theorem}\label{Thm:B}
Let $M^{n+1}_1$ $(n\ge 2)$ be a light-like geodesically 
complete Lorentzian manifold.
Then, any $C^2$-differentiable proper null 
hypersurfaces immersed in $M^{n+1}_1$
are $L$-complete.  
\end{theorem}

\begin{proof}
By Nomizu and  Ozeki \cite{NO}, there exists a 
complete Riemannian metric
$g_R$ on $M^{n+1}_1$. We fix this Riemannian 
metric $g_R$. 
Let $F:\Lambda^n\to M^{n+1}_1$ be
a proper null immersion. 
By taking a universal covering of $\Lambda^n$, we may assume that
$\Lambda^n$ is simply connected.
We let $h_R:=F^*g_R$ be the Riemannian metric on $\Lambda^n$,
which may not be complete in general.

We shall now prove that $h_R$ is a complete metric 
under the assumption that $F$ is proper.
These two metrics $g_R$ and $h_R$ induce 
distance functions $d_M$ on $M^{n+1}_1$
and $d_\Lambda$ on $\Lambda^n$ which are compatible with
respect to the topologies of $M^{n+1}_1$ and  $\Lambda^n$, respectively.
Then the following inequality holds:
\begin{equation}\label{eq:DLE}
d_\Lambda(p,q)\ge d_M(F(p),F(q))\qquad (p,q\in \Lambda^n).
\end{equation}
Let $\{p_j\}_{j=1}^\infty$ be a Cauchy sequence of $(\Lambda^n,d_\Lambda)$.
By \eqref{eq:DLE},
$\{F(p_j)\}_{j=1}^\infty$ is also a Cauchy sequence of $(M^{n+1}_1,d_M)$.
Since $g_R$ is a complete Riemannian metric, 
$\{F(p_j)\}_{j=1}^\infty$ converges to a point $Q\in M^{n+1}_1$.
Consider a geodesic ball $B_Q(r)(\subset M^{n+1}_1)$  with respect to $g_R$
of radius $r$ centered at $Q$.
If $r$ is sufficiently large, $\{F(p_j)\}_{n=j}^\infty$ lies in $B_Q(r)$.
Since $F$ is a proper map, $F^{-1}(\overline{B_Q(r)})$ is a compact
subset of $\Lambda^n$.
Since $\{p_j\}_{j=1}^\infty$ lies in $F^{-1}(\overline{B_Q(r)})$,
the sequence $\{p_j\}_{n=j}^\infty$ has an accumulation point $q$.
Since $\{p_j\}_{j=1}^\infty$ is a Cauchy sequence,
it must converge to a point $q$.
Thus $(\Lambda^n,d_\Lambda)$ is complete as a metric space,
and $h_R$ is a complete Riemannian metric on $\Lambda^n$.
Since $\Lambda^n$ can be assumed to be simply connected,
we can take a null vector field $\xi$ without zeros defined on
$\Lambda^n$.
We then consider the null vector field
$$
\xi_R:=\xi/\sqrt{h_R(\xi,\xi)}.
$$
Since $\xi_R$ is a unit vector field with respect to the complete 
Riemannian metric $h_R$,
it is a complete vector field on $\Lambda^n$.
Then each maximal integral curve $\gamma(t)$ 
of $\xi_R$ is defined on $\R$.
We set $\hat \gamma(t)=F\circ \gamma(t)$.
To complete the proof, it is sufficient to
show the following lemma:
\end{proof}

\begin{lemma}\label{lemma:L1}
Let $M^{n+1}_1$ be a light-like geodesically complete
Lorentzian manifold.
If $\hat\gamma\colon\R\to M^{n+1}_1$ is a  
pre-geodesic such that $g_R(\hat \gamma'(t),\hat \gamma'(t))=1$
for each $t\in \R$, then $\hat \gamma$ is 
$L$-complete. 
\end{lemma}

\begin{proof}
We set $P:=\hat \gamma(0)\in M^{n+1}_1$.
By Fact \ref{lem:F}, 
we have $\hat \gamma(\R)\subset L$,
where $L$ is the image of a complete
light-like geodesic in $M^{n+1}_1$.
Suppose that $\hat \gamma(\R)$ is a
proper subset of $L$. 
We can take a point $Q\in L\setminus \hat \gamma(\R)$.
Then the length of a segment $[P,Q]$
of $L$ (with respect to the metric $g_R$)
bounded by $P$ and $Q$
must be finite.
However, this contradicts that 
the subarc $\hat\gamma([0,\infty))$ of the segment $[P,Q]$
has infinite length with respect to the metric $g_R$.
\end{proof}

\section{Proof of Main results} \label{sec:main} 

In this section, we 
assume that the Ricci tensor of $M^{n+1}_1$ 
is non-negative along light-like directions
$($that is, $M^{n+1}_1$ satisfies the null energy condition$)$.
By Theorem \ref{Thm:B}, to prove Theorem A in
the introduction,
it is sufficient to prove the following assertion,
where we do not assume that
$M^{n+1}_1$ is light-like geodesically complete:

\begin{theorem}\label{thm:old-A}
Let $F\colon\Lambda^n\to M^{n+1}_1$ $(n\ge 2)$
be an $L$-complete null $C^2$-immersion.
Suppose that the Ricci tensor of $M^{n+1}_1$ 
is non-negative along light-like directions,
that is,  $M^{n+1}_1$ satisfies the null energy 
condition $($NEC$)$.
Then, $F$ is totally geodesic
$($that is, any geodesic in $M^{n+1}_1$ which is
tangent to $F(\Lambda^n)$ at a certain point
lies in $F(\Lambda^n))$, and
the Ricci tensor of $M^{n+1}_1$ 
vanishes along null directions of $\Lambda^n$.
\end{theorem}

Let $F\colon\Lambda^n\to M^{n+1}_1$ $(n\ge 2)$
be a $C^2$-differentiable $L$-complete null immersion.
We fix a point $p\in \Lambda^n$ and
a non-zero tangent vector $\mb v\in T_p\Lambda^n$
arbitrarily.
Then there exist a positive number $\epsilon$ and
a unique geodesic 
$\hat \gamma\colon(-\epsilon,\epsilon)\to M^{n+1}_1$ such that 
\begin{equation}\label{eq:v}
\hat \gamma(0)=F(p),\quad \hat \gamma'(0)=dF(\mb v).
\end{equation}

\medskip
\noindent
{({\it Proof of Theorem \ref{thm:old-A}}.)
Since the desired conclusion is local,
we can take a non-vanishing null vector field $\xi$ 
defined on a 
sufficiently small coordinate neighborhood $(U,u_1,\dots,u_n)$ at $p$.
Since the images of integral curves of $\xi$ by $F$ 
are light-like pregeodesics
and light-like pregeodesics turns into geodesics by  
suitable changes of parameter, 
we may assume that $\nabla_\xi\hat \xi=0$.
Here, the $\nabla$ is the 
Levi-Civita connection on $(M^{n+1}_1,g_L)$
 and $\hat{\xi}=dF(\xi)$.
To prove Theorem A, it is sufficient to show that
$\hat \gamma$ lies on the image of $F$.
If $\mb v$ is proportional to $\xi$,
we have already seen that $\hat \gamma$
is  a pre-geodesic of $M^{n+1}_1$ which lies in $F(\Lambda^n)$.
So we may assume that $dF(\mb v)$ is a space-like vector.

Using the coordinate system $(U,u_1,\dots,u_n)$,
one can construct
space-like vector fields $S_1, \dots, S_{n-1}$ 
(called screen vector fields)
such that
$$
S_1,\,\, \dots,\,\, S_{n-1},\,\, \xi
$$
give a frame field of $\Lambda^n$ on $U$.
Then we can take a unique vector field $\eta$ of $M^{n+1}_1$
defined on $U$ such that
$$
g_L(\eta,\eta)=0,\quad
g_L(\hat \xi,\eta)=1,\quad g_L(\hat S_i,\eta)=0
\qquad
(i=1,...,n-1),
$$
where $\hat S_i:=dF(S_i)$ $(i=1,\dots,n-1)$.
Using this, we can define a torsion-free connection $D$
on $U$ satisfying
$$
\nabla_{dF(X)}dF(Y)=dF(D_XY)+B(X,Y)\eta
$$
for each pair of vector fields $X,Y$ on $U$,
where $B(X,Y)$ is a covariant symmetric tensor defined on $U$.
Then, there exists a $D$-geodesic
(i.e. a geodesic with respect to $D$)
 $\gamma\colon[-\epsilon_0,\epsilon_0]\to U$
such that $\gamma(0)=p$ and $\gamma'(0)=\mb v$,
where $\epsilon_0>0$.
We can choose $\epsilon_0$ so that $\gamma$ is an embedding.
By a standard orthogonalization procedure,
we can take linearly independent vector
fields 
$
V_0,\,\,V_1,\,\,\dots,\,\,V_{n-2}
$
along $\gamma$ such that 
\begin{itemize}
\item[(a)]  
$V_0(t)$
coincides with $\gamma'(t)$,
and each $V_i(t)$ ($i=0,...,n-2$)
is expressed as a linear combination of 
$\{S_i\}_{i=1}^{n-1}$ at $\gamma(t)$,
\item[(b)]
$g_L\biggl(dF(V_0(t)),
dF(V_j(t))\biggr)=0$
for $j=1,...,n-2$
and $t\in (-\epsilon_0, \epsilon_0)$,
and
\item[(c)]
$g_L\biggl(dF(V_i(t)),
dF(V_j(t))\biggr)=\delta_{ij}$
holds for $i,j=1,...,n-2$
and $t\in (-\epsilon_0, \epsilon_0)$, where $\delta_{ij}$ 
is the Kronecker's delta.
\end{itemize}

It is sufficient to prove $F\circ \gamma=\hat \gamma$.
For a sufficiently small $\delta(>0)$,
we set
$$
\Omega:=\{(t_0,t_1,...,t_{n-2},u)\in \R^{n}\,;\, |t_0|<\epsilon_0, \,\,\,
|u|,\,|t_i|<\delta \,\,\, (i=1,...,n-2)\}.
$$
We define a map $\phi\colon\Omega\to \Lambda^n$ by
$$
\phi(t_0,t_1,...,t_{n-2},u):=\op{Exp}^D_{\gamma(t_0)}
\left (u\xi_{\gamma(t_0)}+\sum_{i=1}^{n-2}t_i  V_i(t_0) \right),
$$
where $(t_0,t_1,...,t_{n-2},u)\in \Omega$ and
$\op{Exp}^D_q$ is the exponential map with respect to 
the connection $D$ centered at $q\in \Lambda^n$.
Since $\gamma$ is an embedding,
$\phi$ gives a local coordinate system of $\Lambda^n$.
For $u\in(-\delta,\delta)$,
we set
$$
\Omega_{u}:=\{(t_0,t_1,...,t_{n-2},u)\in \Omega\,;
\,|t_0|<\epsilon_0,\, |t_i|<\delta \,\,\, (i=1,...,n-2)
\},
$$
and denote by $\phi_{u}$ the restriction of $\phi$ to $\Omega_u$.
Then $\{\phi_u(\Omega_u)\}_{|u|<\delta}$ is 
a family of embedded hypersurfaces in $\Lambda^n$
which is transversal to $\xi$. 
We fix $u\in (-\delta,\delta)$ arbitrarily, 
and let
$$
\Psi_u:\R\times \Sigma^{n-1}_u\to \Lambda^{n}
$$
be the map given in Proposition \ref{prop:Phi}
by setting $\Sigma^{n-1}_u:=\phi_u(\Omega_{u})$.
Then $\hat \Psi_u:=F\circ \Psi_u$ is a $L$-complete 
null immersion.
Since $M^{n+1}_1$ satisfies
the null energy condition,
the mean curvature vector 
field  $\mb H_u(t_0,t_1,...,t_{n-2})$ of $F(\Sigma^{n-1}_u)$
satisfies,  by \cite[Prop. 43 in Chap. 10]{ON},
\begin{equation}\label{eq:H}
(H_u:=)g_L(\mb H_u,\hat\xi)
=0.
\end{equation}
In fact, if we set $\sigma(t):=\hat{\Psi}_u(t,q)$
($q\in \Sigma^{n-1}_u$) and if
$H_u=g_L(\mb H_u,\sigma'(0))\ne 0$, then
there exists a focal point of $F(\Sigma^{n-1}_u)$
on $\sigma(\R)$. However, this contradicts the
expression \eqref{eq:exp}, since $F\circ \Psi_u$
is an immersion. 
It can be easily checked that the function
$$
(t_0,t_1,...,t_{n-2},u)\mapsto H_u(t_0,t_1,...,t_{n-2})
$$ 
defined on $\Omega$
coincides with the light-like mean curvature function $H_\xi$
with respect to $F\circ \phi$
up to a constant multiple 
(cf. \cite[(II.1)]{Ga2} and \cite{DG}). 
Since $H_u=0$ for each $u$,
by the Raychaudhuli equation
(cf. \cite[(13)]{DG}, \cite[Theorem 45]{K} and \cite[(II.4)]{Ga2}),
we have that
$$
\op{Ric}(\hat\xi,\hat\xi)+\op{Trace}(A_\xi^2)=0,
$$
which yields that $\op{Ric}(\hat \xi,\hat \xi)$ and
$A_\xi(V_i)$ $(i=1,...,n-1)$ vanishes along the curve $\sigma$,
where $\op{Ric}$ is the Ricci tensor of $M^{n+1}_1$ and
$A_\xi$ is the shape operator
on the distribution spanned by $\{S_1,...,S_{n-1}\}$.
We have that
\begin{align*}
B(V_i,V_j)
=g_L(\nabla_{V_i}V_j,\xi)
=-g_L(V_j,\nabla_{V_i}\xi)=g_L(V_j,A_\xi(V_i))=0.
\end{align*}
Since $B(\xi,X)=0$
for any tangent vector $X$ of $\Lambda^n$,
we have $B=0$ 
along $\gamma$. 
By \cite[Corollary 47]{K}, this fact yields that $F\circ\gamma$ 
is a geodesic of $M^{n+1}_1$. 
(In fact, since $\gamma$ is a $D$-geodesic,
$
\nabla_{(F\circ \gamma)'}(F\circ \gamma)'=
dF(D_{\gamma'}\gamma')+B(\gamma',\gamma')
=B(\gamma',\gamma')=0.
$)
Thus, we can conclude that $F\circ\gamma=\hat\gamma$.
\qed

\begin{proof}[Proof of Corollary C]
Since the graph of an entire function $f$ is 
properly embedded,
the assertion follows from Theorem A, Theorem \ref{Thm:B}
and  Corollary B. 
\end{proof}

We give here several examples:

\begin{example}
We set
$$
S^3_1:=\{(t,x,y,z)\in \R^4_1\,;\, -t^2+x^2+y^2+z^2=1\},
$$
which gives the  de Sitter 3-space of constant curvature $1$. 
Then
$$
F(s,t):=(t,\cos s,\sin s,t)
$$
gives an  $L$-complete
totally geodesic null surface in $S^3_1$.
\end{example}

\begin{example}
Similarly, we set
$$
H^3_1:=\{(t,x,y,z)\in \R^4_2\,;\, -t^2-x^2+y^2+z^2=1\},
$$
which gives the  anti-de Sitter 3-space of constant curvature $-1$. 
Then
$$
F(s,t):=(t,\cosh s,\sinh s,t)
$$
gives an  $L$-complete totally  geodesic null surface in $H^3_1$.
\end{example}

\begin{example}
We set 
$$
M^3_1:=\{(t,x,y,z)\in \R^4_2\,;\, -x^2+y^2+z^2=-1,\,\, x>0\}.
$$
Then $M^3_1$ is a product $\R\times H^2$, where
$H^2$ is a hyperbolic plane. An embedding defined by
$$
F(s,t):=(t,\cosh s \cosh t,\sinh s \cosh t,\sinh t)
$$
is  $L$-complete and null.
However, $F$ is not totally geodesic.
In fact, 
$M^3_1$ does not satisfy the null energy condition.
\end{example}

\appendix
\section{A lemma on complete vector fields}

We recall the definition of completeness of vector
fields as follows:

\begin{definition}\label{def:vf}
Let $X$ be a vector field defined on an $n$-manifold $M^n$.
Then $X$ is called {\it complete} if for each $p\in M^n$,
there exists a curve
$\gamma\colon\R\to M^n$ such that $\gamma(0)=p$
and $\gamma'(t)=X_{\gamma(t)}$ for $t\in \R$.
\end{definition}

The following fact is well-known (cf. \cite[Theorem 2.95]{L}):

\begin{fact}\label{fact:Diff}
Let $X$ be  a complete vector field defined on $M^n$.
Then it induces a $1$-parameter group of transformations
$\phi_t\colon M^n\to M^n$ such that
$\R\ni t \mapsto \phi_t(p)\in M^n$
is an integral curve of $X$ passing through $p\in M^n$ at $t=0$. 
\end{fact}

We prove the following fact: 

\begin{proposition}\label{prop:Phi2}
Let $X$ be a complete vector field defined on $M^n$. 
Let $\Sigma^{n-1}$  be an embedded hypersurface of $M^n$.
Suppose that $X$ is transversal to $\Sigma^{n-1}$.
Then the map defined by
$$
\Phi\colon\R\times \Sigma^{n-1}\ni (t,p)\to \phi_t(p)\in M^n
$$
gives an immersion.
\end{proposition}

\begin{proof}
Since each $\phi_t$ is a diffeomorphism on $M^n$,
the restriction of $\phi_t$ to $\Sigma^{n-1}$ 
is an immersion for each $t$.
In particular, the rank of
Jacobi matrix of $\Phi$ is greater than or equal
to $n-1$ at $p$. 
Since $X$ is transversal to $\Sigma^{n-1}$,
the fact that $\phi_t\colon M^n\to M^n$
is a diffeomorphism yields that
$$
X_{\phi_t(p)}=\left.\frac{d\phi_{t+s}(p)}{ds}\right|_{s=0}=(d\phi_t)_p(X_p)
$$
is transversal to $\phi_t(\Sigma^{n-1})$.
Thus, the rank of Jacobi matrix of $\Phi$ must be equal
to $n$ at $(t,p)$.
\end{proof}

\begin{acknowledgement}
The authors thank Professors 
Shyuichi Izumiya and Alfonso Romero
for valuable information.
\end{acknowledgement}

\end{document}